\documentclass[12pt,psfig,reqno]{amsart}
\usepackage{mathrsfs}
\usepackage{txfonts}
\usepackage{amsmath}
\usepackage{bbm}
\usepackage{}
\usepackage{amssymb}
\usepackage{amsfonts}
\usepackage{amscd}
\usepackage{amssymb,amsfonts,latexsym}
\usepackage{graphics,verbatim}
\usepackage{graphicx}
\usepackage[usenames]{color} 
\makeatletter

\renewcommand\section{\@startsection {section}{1}{\z@}%
                                   {-3.5ex \@plus -1ex \@minus -.2ex}%
                                   {2.3ex \@plus.2ex}%
                                   {\centering\normalfont\bf}}

 \numberwithin{equation}{section}
\numberwithin{equation}{section}

\numberwithin{equation}{section}

\theoremstyle{plain}
\newtheorem{thm}{Theorem}[section]
\newtheorem{lemma}[thm]{Lemma}
\newtheorem{pro}[thm]{Proposition}

\newtheorem{de}[thm]{Definition}
\newtheorem{re}[thm]{Remark}
\newtheorem*{thm*}{Theorem}

\begin{document}
\title{The cardinality of orthogonal exponentials of planar self-affine measures with three-element digit sets}
\author{Ming-Liang Chen and Jing-Cheng Liu$^*$}
\address{Key Laboratory of High Performance Computing and Stochastic Information Processing (Ministry of Education of China), College of Mathematics and Statistics, Hunan Normal University, Changsha, Hunan 410081, China}
\email{jcliu@hunnu.edu.cn}
\email{mathcml@163.com}

\date{\today}
\thanks{ The research is supported in part by the NNSF of China (Nos. 11301175 and 11571104),
 the SRF of Hunan Provincial Education Department (No.17B158).\\
$^*$Corresponding author.}

\begin{abstract}
In this paper, we consider the planar self-affine measures $\mu_{M,D}$ generated by  an expanding matrix $M\in M_2(\mathbb{Z})$ and an integer digit set
$
D=\left\{ {\left( {\begin{array}{*{20}{c}}
0\\
0
\end{array}} \right),\left( {\begin{array}{*{20}{c}}
\alpha_1\\
\alpha_2
\end{array}} \right),\left( {\begin{array}{*{20}{c}}
\beta_1\\
\beta_2
\end{array}} \right)} \right\}
$
with $\alpha_1\beta_2-\alpha_2\beta_1\neq0$. We show that if $\det(M)\notin 3\mathbb{Z}$, then the  mutually orthogonal exponential functions in $L^2(\mu_{M,D})$ is finite, and the exact maximal cardinality is given.
\end{abstract}

\maketitle

\section{\bf Introduction\label{sect.1}}

Let $M\in M_n(\mathbb{R})$ be an expanding  matrix(that is, all the eigenvalues of $M$ have moduli $>1$), and let $D\subset\mathbb{R}^n$ be a finite subset with cardinality $|D|$. Let $\{\phi_d\}_{d\in D}$ be an iterated function system (IFS) on $\mathbb{R}^n$ defined by
\begin{equation*}\label {eq(1.1)}
\phi_d(x)=M^{-1}(x+d)   \ \  (x\in \mathbb{R}^n,  \ \ d\in D).
\end{equation*}
Then the IFS arises a natural \emph{self-affine measure} $\mu:=\mu_{M,D}$ satisfying
\begin{equation}\label {eq(1.1)}
\mu=\frac{1}{|D|}\sum_{d\in D}\mu\circ\phi_d^{-1}.
\end{equation}
The measure $\mu_{M,D}$ is supported on the attractor of the IFS $\{\phi_d\}_{d\in D}$\cite{Hutchinson_1981}.

For a countable subset $\Lambda\subset\mathbb{R}^n$, let $\mathcal {E}_\Lambda=\{e^{2\pi i\langle\lambda,x\rangle}:\lambda\in\Lambda\}$. We call $\mu$  a spectral measure, and $\Lambda$ a  spectrum of $\mu$ if $\mathcal {E}_\Lambda$ is an orthogonal basis for $L^2(\mu)$. We also say that $(\mu,\Lambda)$ is a $spectral$ $pair$. The existence of a spectrum for $\mu$ is a basic problem in harmonic analysis, it was initiated by Fuglede in his seminal paper \cite{Fuglede_1974}. After the original work of Fuglege, the spectral problem has been investigated in a variety of different
mathematical fields. The first example of a singular, non-atomic, spectral measure which is supported on $\frac{1}{4}$-Cantor set was given by Jorgensen and Pedersen in \cite{Jorgensen-Pedersen_1985}. This surprising discovery received a lot of attention, and the spectrality of self-affine measures has become a hot topic. Many new spectral measures  were found  in \cite{Dai_2012}-\cite{Dutkay-Haussermann-Lai_2015}, \cite{Hu-Lau_2008}, \cite{Li_2011}-\cite{Li_2012}, \cite{Liu-Luo_2017} and references therein. For more general cases such as Moran measures, the reader can refer to \cite{An-He_2014}-\cite{An-He-Lau_2015}.

 \par On the other hand, the non-spectral problem of  self-affine measures is also very interesting.  In \cite{Dutkay-Jorgensen_2007}, Dutkay and Jorgensen showed that if
$
M=\left[ {\begin{array}{*{20}{c}}
{{p}}&{{0}}\\
{{0}}&{{{p}}}
\end{array}} \right]
$ with $p\in \mathbb{Z}\setminus 3\mathbb{Z}$, $p\geq 2$ and $\mathscr{D}=\{(0,0)^t, (1,0)^t, (0,1)^t\}$,
then there are no 4 mutually orthogonal exponential functions in $L^2(\mu_{M,\mathscr{D}})$. Moreover, they  proved that if
$
M=\left[ {\begin{array}{*{20}{c}}
{{2}}&{{1}}\\
{{0}}&{{{2}}}
\end{array}} \right],
$
then there exist at most $7$ mutually orthogonal exponential  functions in $L^2(\mu_{M,\mathscr{D}})$.
In \cite{Li_2008}, Li proved that if  the expanding integer matrix
$
M=\left[ {\begin{array}{*{20}{c}}
{{a}}&{{b}}\\
{{0}}&{{{c}}}
\end{array}} \right]
$ with $\det(M)\not\in 3\mathbb{Z}$, then there exist at most $3$ mutually orthogonal exponential  functions in $L^2(\mu_{M,\mathscr{D}})$, and the number 3 is the best.
More recently, Liu, Dong and Li \cite{Liu-Dong-Li_2017} extended the above upper triangular matrix  to
$M=\left[ {\begin{array}{*{20}{c}}
{{a}}&{{b}}\\
{{d}}&{{{c}}}
\end{array}} \right]$ with $\det(M)\not\in 3\mathbb{Z}$,
and proved that there exist at most $9$ mutually orthogonal exponential  functions in $L^2(\mu_{M,\mathscr{D}})$, and the number $9$ is the best.

Let $D=\{(0,0)^t,(\alpha_1, \alpha_2)^t, (\beta_1,\beta_2)^t\}$, if we assume that $\alpha_1\beta_2-\alpha_2\beta_1=1$, it can be easily seen that there exist at most $9$ mutually orthogonal exponential  functions in $L^2(\mu_{M,D})$  by Theorem 1.1 of \cite{Liu-Dong-Li_2017}. A natural question is whether the number $9$ is suitable for any three-element integer digit set? Motivated by the previous research, we will give a complete answer in this paper. Without loss of generality, we may assume that $\gcd(\alpha_1,\alpha_2,\beta_1,\beta_2)=1$ by Lemma \ref{th(2.2)}.

\begin{thm} \label{th(1.1)} For an expanding matrix $M\in M_2(\mathbb{Z})$  with $\det(M)\notin 3\mathbb{Z}$ and an integer digit set  $D=\{(0,0)^t,(\alpha_1,\alpha_2)^t,(\beta_1,\beta_2)^t\}$, let $\mu_{M,D}$ be defined by \eqref{eq(1.1)}. The following hold.

(i)  If $2\alpha_1-\beta_1\notin3\mathbb{Z}$ or $2\alpha_2-\beta_2\notin3\mathbb{Z}$, then there exist at most 9 mutually orthogonal exponential functions in $L^2(\mu_{M,D})$, and the number 9 is the best.

(ii) If $2\alpha_1-\beta_1, 2\alpha_2-\beta_2$
$\in3\mathbb{Z}$, then there exist at most $3^{2\eta}$ mutually orthogonal exponential functions in $L^2(\mu_{M,D})$, and the number $3^{2\eta}$ is the best, where $\eta=\max\{r:3^r|(\alpha_1\beta_2-\alpha_2\beta_1)\}$.
\end{thm}

The case (ii) of Theorem \ref{th(1.1)} actually follows from a more general result. Before stating the result, we need some definitions and notations.

\par
For positive integers $p,q\geq 2$ and $s\geq1$, let
\begin{equation}\label{eq(1.2)}
E_q^n:=\frac{1}{q}\{(l_1,l_2,\cdots,l_n)^t:0\leq l_1,\cdots,l_n\leq q-1, l_i\in\mathbb{Z}\}, \ \ \ \ \ \ \ \mathring{E}_q^n:=E_q^n\setminus\{0\}
\end{equation}
and
\begin{equation}\label{eq(1.3)}
\mathcal {A}_p(s):=\frac{1}{p^s}\{(p^{s-1},l)^t:0\leq l\leq p^s-1, l\in \mathbb{Z}\}.
\end{equation}
For a finite digit set $D\subset\mathbb{R}^n$, let
\begin{equation}\label{eq(1.4)}
m_{D}(x)=\frac{1}{|D|}\sum\limits_{d\in D}{e^{2\pi i\langle d,x\rangle}}, \quad x\in\mathbb{R}^n, \ \ \ \ \ \ \ \  \mathcal {Z}(m_{D}):=\{ x\in \Bbb R^n:\;m_{D}(x)=0 \},
\end{equation}
where $m_{D}(x)$ is called the\emph{ mask polynomial} of $D$ as usual.
Define
\begin{equation}\label{eq(1.5)}
\mathcal{Z}_{D}^n:=\mathcal{Z}(m_{D})\cap [0, 1)^n.
\end{equation}
It is easy to see that $m_D$ is a $\mathbb{Z}^n$-periodic function if $D\subset \mathbb{Z}^n$. In this case,
$\mathcal{ Z}(m_D)=\mathcal{Z}_{D}^n+\mathbb{Z}^n$.

\begin{thm} \label{th(1.2)}
Assume integers $p\geq 2,\bar{\eta}\geq1$ and a finite digit set $D\subset\mathbb{R}^2$. Let $M\in M_2(\mathbb{Z})$ be an expanding matrix with $\gcd(\det(M),p)=1$, and let
$\mu_{M,D}$, $\mathring{E}_{p^{\bar{\eta}}}^2$, $\mathcal {A}_p({\bar{\eta}})$, $\mathcal{ Z}(m_D)$ be defined by \eqref{eq(1.1)}, \eqref{eq(1.2)}, \eqref{eq(1.3)} and \eqref{eq(1.4)}, respectively. If $\mathcal{ Z}(m_D)\subset \mathring{E}_{p^{\bar{\eta}}}^2+\mathbb{Z}^2$, then there exist at most $p^{2{\bar{\eta}}}$ mutually
orthogonal exponential functions  in $L^2(\mu_{M,D})$.  Moreover, if $p\geq 3$ is a prime and there exists $\mathcal{N} \in \mathbb{Z}\setminus p\mathbb{Z}$ such that $\mathcal{N}(\mathcal {A}_p({\bar{\eta}})+\mathbb{Z}^2)\subset \mathcal{ Z}(m_D)$, then the number $p^{2{\bar{\eta}}}$ is the best.
\end{thm}

We arrange the paper as follows. In Section 2, we recall a few basic concepts and notations, establish several lemmas that will be needed in the proof of our main results. In Section 3, we give the detailed proofs of Theorems \ref{th(1.1)} and \ref{th(1.2)}.

\section{\bf Preliminaries\label{sect.2}}

In this section, we give some  preliminary definitions and  lemmas. We will start with an introduction to the Fourier transform. For a $n\times n$ expanding real matrix $M$ and a finite digit set $D\subset\mathbb{R}^n$, let $\mu_{M,D}$ be defined by \eqref{eq(1.1)}. The Fourier transform
$\hat{\mu}_{M,D}(\xi)=\int e^{2\pi i\langle x,\xi\rangle }d\mu_{M,D}(x)~(\xi\in\mathbb{R}^n)$ of $\mu_{M,D}$ plays an important role in the study of the spectrality of $\mu_{M,D}$. It follows from \cite{Dutkay-Jorgensen_2007} that
\begin{equation}\label {eq(2.1)}
\hat{\mu}_{M,D}(\xi)=\prod_{j=1}^\infty m_D({M^{*}}^{-j}\xi), \ \ \ \ \  \xi\in\mathbb{R}^n,
\end{equation}
where $M^*$ denotes the transposed  conjugate of $M$, and
$$
m_D(x)=\frac{1}{|D|}\sum\limits_{d\in D}{e^{2\pi i\langle d,x\rangle}},\ \ \ \ \ x\in\mathbb{R}^n.
$$
 For any $\lambda_1, \lambda_2\in \mathbb{R}^n$, $\lambda_1\neq \lambda_2$, the orthogonality condition
$$
0=\langle e^{2\pi i \langle \lambda_1,x\rangle},e^{2\pi i \langle \lambda_2,x\rangle}\rangle_{L^2(\mu_{M,D})}=\int e^{2\pi i \langle\lambda_1-\lambda_2,x\rangle}d\mu_{M,D}(x)=\hat{\mu}_{M,D}(\lambda_1-\lambda_2)
$$
relates to the zero set $\mathcal{Z}(\hat{\mu}_{M,D})$ directly. It is easy to see that for a countable subset $\Lambda\subset\mathbb{R}^n$,  $\mathcal {E}_\Lambda=\{e^{2\pi i\langle\lambda,x\rangle}:\lambda\in\Lambda\}$ is an orthonormal family of $L^2(\mu_{M,D})$ if and only if
 \begin{equation}\label {eq(2.2)}
 (\Lambda-\Lambda)\setminus\{0\}\subset\mathcal{ Z}(\hat{\mu}_{M,D}).
\end{equation}
From \eqref{eq(2.1)}, we have
$\mathcal{Z}(\hat{\mu}_{M,D})=\Big\{\xi\in \mathbb{R}^n: \exists\; j\in \mathbb{N} \ {\rm {such  \ that}}\ m_D(M^{*-j}\xi)=0\Big\}$. Hence
\begin{equation}\label {eq(2.3)}
\mathcal{ Z}(\hat{\mu}_{M,D})=\bigcup_{j=1}^{\infty}M^{*j}(\mathcal{ Z}(m_D)),
\end{equation}
where $\mathcal{ Z}(m_D)=\{x\in\mathbb{R}^n:m_D(x)=0\}$.

\begin{de}\label{de(2.1)}
Let  $\mu$ be a Borel probability measure with compact support on $\mathbb{R}^n$.  Let $\Lambda$ be a finite or a countable subset of $\mathbb{R}^n$, and let ${\mathcal E}_{\Lambda}=\{e^{2\pi i\langle\lambda,x\rangle}:\lambda\in\Lambda\}$. We denote ${\mathcal E}_{\Lambda}$ by  ${\mathcal E}^*_{\Lambda}$ if ${\mathcal E}_{\Lambda}$ is a maximal orthogonal set of exponential functions in $L^2(\mu)$. Let
\begin{align}\label{eq(2.4)}
n^*(\mu):=\sup\{ |\Lambda| :  \mathcal {E}_{\Lambda}^* \hbox { is a maximal orthogonal set} \},
\end{align}
and call $n^*(\mu)$  the maximal cardinality of the orthogonal exponential functions  in $L^2(\mu_{M,D})$.
\end{de}

The following lemma indicates that  the spectral properties of
$\mu_{M,D}$ are invariant under a linear transform. The proof is the same as that of Lemma 4.1 of {\rm \cite{Dutkay-Jorgensen_2007}}.

\begin{lemma}\label{th(2.2)}
Let $D,\widetilde{D}\subset \Bbb R^n$ be two finite digit sets with the same cardinality, and $M,\widetilde{M}\in M_n(\Bbb R)$ be two expanding matrices. If there exists a matrix  $Q\in M_n(\mathbb{R})$  such that $\widetilde{D}=QD$ and $\widetilde{M}=QMQ^{-1}$, then

 {\rm(}i{\rm)}\;\;a set $\Lambda\subset \Bbb R^n$ is an orthogonal set for $\mu_{M,D}$ if and only if $\widetilde{\Lambda}:=Q^{*-1}\Lambda$ is an orthogonal set for $\mu_{\widetilde{M},\widetilde{D}}$. In particular, $n^*(\mu_{M,D})=n^*(\mu_{\widetilde{M},\widetilde{D}})$;

{\rm (}ii{\rm)}\;the $\mu_{M,D}$ is a spectral measure with spectrum $\Lambda$ if and only if the $\mu_{\widetilde{M},\widetilde{D}}$ is a spectral measure with spectrum $Q^{*-1}\Lambda$.
\end{lemma}

For any three-element digit set $D=\{(0,0)^t,(\alpha_1, \alpha_2)^t, (\beta_1,\beta_2)^t\}$ and an expanding matrix $M$ with $\gcd(\det(M),3)=1$, if there exists an invertible matrix $Q$ such that $\mathscr{D}=QD=\{(0,0)^t,(1, 0)^t, (0,1)^t\}$ and  $\widetilde{M}=QMQ^{-1}$ is  an expanding integer matrix, then the following lemma can be used to judge the maximum number of the orthogonal exponential functions in $L^2(\mu_{M,D})$ by lemma \ref{th(2.2)}.

\begin{lemma}  \cite [Corollary 4.1]{Liu-Dong-Li_2017} \label{th(2.3)}\;\;For an expanding matrix $\widetilde{M}\in M_2(\mathbb{Z})$ and a digit set  $\mathscr{D}=\{(0,0)^t,(1,0)^t,(0,1)^t\}$, let $\mu_{\widetilde{M},\mathscr{D}}$ be defined by \eqref{eq(1.1)}. If $\det(\widetilde{M})\notin3\mathbb{Z}$, then there exist at most 9 mutually orthogonal exponential functions in $L^2(\mu_{\widetilde{M},\mathscr{D}})$, and the number 9 is the best.

\end{lemma}

In \cite[Theorem 3.1]{Dutkay-Jorgensen_2007},  Dutkay and Jorgensen established a criterion for the  non-spectrality of
self-affine measures $\mu_{M,D}$, which require that the elements of matrix $M$ and digit set $D$ are all integers.  The following lemma is a little generalization and  can be proved similarly.

\begin{lemma}\label{th(2.4)}\;\; For a $n\times n$ expanding integer matrix $M$ and a finite digit set $D\subset \mathbb{R}^n$, let
 $\mu_{M,D}$, $\mathcal{Z}(m_D)$ be defined by \eqref{eq(1.1)} and \eqref{eq(1.4)}, respectively. If there exists a set $\emptyset\neq\mathrm{Z}'\subset [0, 1)^n$ with finite cardinality $|\mathrm{Z}'|$, which does not
contain $0$, such that $\mathcal{Z}(m_D)\subset\mathrm{Z}'+\mathbb{Z}^n$ and
$$
M^{*}(\mathrm{Z}'+\mathbb{Z}^n)\subset \mathrm{Z}'+\mathbb{Z}^n ,
$$
then there exist at most $|\mathrm{Z}'|+1$ mutually orthogonal exponential functions in $L^2(\mu_{M,D})$. In particular, $\mu_{M,D}$ is not a spectral measure.
\end{lemma}
\begin{proof}Suppose there exists a family of mutually orthogonal exponential functions $\{e^{2\pi i\langle\lambda,x\rangle}:\lambda\in\Lambda\}$ with $\mid\Lambda\mid > |\mathrm{Z}'|+1$. By taking some $\lambda_0\in\Lambda$ and replacing $\Lambda$ by $\Lambda-\lambda_0$, we may assume that $0\in\Lambda$. For any $\lambda_1\ne\lambda_2\in\Lambda$, the orthogonality implies that $\hat{\mu}_{M,D}(\lambda_1-\lambda_2)=0$. By \eqref{eq(2.3)}, there exists an integer $k\ge 1$ such that $M^{*-k}(\lambda_1-\lambda_2)\in \mathcal{Z}(m_D)\subset  \mathrm{Z}'+\mathbb{Z}^n$. By the hypothesis, we get  $\lambda_1-\lambda_2\in M^{*k}(\mathcal{Z}(m_D))\subset  M^{*k}(\mathrm{Z}'+\mathbb{Z}^n)\subset\mathrm{Z}'+\mathbb{Z}^n$, and hence
\begin{equation}\label{eq(2.5)}
\Lambda\setminus \{0\}\subset (\Lambda-\Lambda)\setminus \{0\}\subset\mathrm{Z}'+\mathbb{Z}^n .
\end{equation}
 Since $\mid\Lambda\mid > |\mathrm{Z}'|+1$, there exist $\lambda_1\ne\lambda_2\in\Lambda$ such that $\lambda_1-\lambda_2\in \mathbb{Z}^n$ by using pigeonhole principle. But this will contradict \eqref{eq(2.5)}, because $0\notin \mathrm{Z}'$.
\end{proof}

\begin{lemma}\label{th(2.5)}{\rm \cite[Proposition 2.2]{Liu-Dong-Li_2017}}\;\;
For an integer $q\geq 2$, let $\mathring{E}_q^2$ be
defined by \eqref{eq(1.2)} and  $M\in M_2(\mathbb{Z})$ with
$\gcd(\det(M),q)=1$. Then $M(q \mathring{E}_q^2)=
q\mathring{E}_q^2\;({\rm mod} \ q\Bbb Z^2)$, equivalently, $M(
\mathring{E}_q^2)= \mathring{E}_q^2\;({\rm mod} \ \Bbb Z^2)$.
\end{lemma}
\begin{re}\label{th(2.6)}
Let $p$, $s\geq 2$ be integers and $M\in M_2(\mathbb{Z})$ with
$\gcd(\det(M),p)=1$. By Lemma \ref{th(2.5)}, we have $M(
\mathring{E}_{p^s}^2)= \mathring{E}_{p^s}^2\;({\rm mod} \ \Bbb Z^2)$ and $M(\mathring{E}_{p^{s-1}}^2)= \mathring{E}_{p^{s-1}}^2\;({\rm mod} \ \Bbb Z^2)$. This implies that $M(\mathring{E}_{p^s}^2\setminus\mathring{E}_{p^{s-1}}^2 )= \mathring{E}_{p^s}^2\setminus\mathring{E}_{p^{s-1}}^2\;({\rm mod} \ \Bbb Z^2)$. Therefore, for any $\xi=\frac{1}{p^s}(l_1,l_2)^t\in\mathring{E}_{p^s}^2$ with $p\nmid\gcd(l_1,l_2)$, $M\xi=
\frac{1}{p^s}(l'_1,l'_2)^t$ must satisfy $p\nmid\gcd(l'_1,l'_2)$.
\end{re}\label{th(3.2)}
For a positive number $m$, let $\varphi(m)$ denote the \emph{Euler's phi function} which  equal to the number of integers in the set $\{1,2,\cdots, m-1\}$ that are relatively prime to $m$.  For more information about the Euler's phi function, the reader can refer to \cite{Nathanson_1996}. The following lemma is the famous \emph{Euler's theorem}.

\begin{lemma}\cite[Theorem 2.12]{Nathanson_1996}\label{th(2.7)}\;\; Let $m$ be a positive integer, and let $N$ be an integer
relatively prime to $m$. Then $N^{\varphi(m)}=1\;({\rm mod}\;m).$
\end{lemma}

For a prime $p$, let $\mathbb{F}_p:=\mathbb{Z}/p\mathbb{Z}$ denote  the residue class fields and $\mathbb{F}^n_p$ denote the vector space of dimension $n$ over $\mathbb{F}_p$. All nonsingular $n\times n$ matrices over $\mathbb{F}_p$ form a finite group under matrix multiplication,
called the \emph{general linear group} $GL(n,\mathbb{F}_p)$.

\begin{de}\label{de(2.8)}
Let $f(x)\in \mathbb{F}_p[x]$ be a nonzero polynomial. If $f(0)\neq0$, then
the least positive integer $n$ for which $f(x)$ divides $x^n-1$ is called the order of
$f$ and denoted by $ord_p(f)$.
\end{de}

The order of the polynomial $f$ is sometimes also called the period of $f$ or the exponent of $f$. There are many conclusions about the order of polynomial in the third chapter of \cite{Lidel-Niederreiter_1994}.
\begin{de}\label{de(2.9)}
Let $M\in GL(n,\mathbb{F}_p)$, then the least positive integer $e$ for which $M^e=I$ is called the order of
$M$ and denoted by $O_p(M)$, where $I$ is the  identity matrix in $GL(n,\mathbb{F}_p)$.
\end{de}

The following lemma reflects the relationship between the order of the matrix $M$ and the order of the characteristic polynomial of $M$.
\begin{lemma}\cite[Proposition 3.1]{Ghorpade-Hasan-Kumari_2011}\label{th(2.10)}
Let $M\in GL(n,\mathbb{F}_p)$ and let $f(x)\in \mathbb{F}_p[x]$ be the minimal polynomial of $M$ and $\chi(x)\in \mathbb{F}_p[x]$ be the characteristic polynomial of $M$. Then $f(0)\neq 0$
and $O_p(M) = ord_p(f)$. In particular, $O_p(M)\leq p^n-1$, and moreover, if the equality
holds, then $f(x)=\chi(x)$. Also, we have:
$O_p(M)= p^n-1$ $\Longleftrightarrow $  $f(x)$ is primitive of degree $n$ $\Longleftrightarrow $ $\chi(x)$ is primitive.
\end{lemma}

It is well known that there exist $\varphi(p^n-1)/n$ primitive polynomials with degree $n$ over $\mathbb{F}_p$(see $P_{87}$
Theorem 4.1.3 of \cite{Mullen-Panario_2013} ), where $\varphi$ is the  Euler's phi function. Consequently, Lemma \ref{th(2.10)} implies that the matrix $M\in GL(n,\mathbb{F}_p)$ with
$O_p(M)= p^n-1$ always exists.

\begin{de}\label{de(2.11)}
We call $M\in GL(n,\mathbb{F}_p)$ an ergodic matrix if $\{Mv,M^2v,\cdots,M^{p^n-1}v\}=p\mathring{E}_p^n\;({\rm mod} \ p\Bbb Z^n)$ for any $v\in \mathbb{F}^n_p\setminus \{0\}$.
\end{de}

The  ergodic matrices have been widely used in Cryptography. The following lemma shows that the ergodic matrices attain to the maximum order of matrices in
$GL(n,\mathbb{F}_p)$.

\begin{lemma}\cite[Lemma 3.3]{Zhao-Zhao-Pei_2012}\label{th(2.12)}
A matrix $M\in GL(n,\mathbb{F}_p)$ is ergodic if and only if $O_p(M)= p^n-1$.
\end{lemma}

\section{\bf Main Results\label{sect.3}}

In this section, we first  prove Theorem \ref{th(1.2)}, and then  prove Theorem \ref{th(1.1)} by using the result of Theorem \ref{th(1.2)}. In the proof of Theorem \ref{th(1.2)}, the ``at most" is  easy to get by Lemma \ref{th(2.4)}, the main difficulty is to show that the number $p^{2{\bar{\eta}}}$ is the best. In order to get this, we will prove that there exists an expanding integer matrix $M$ with $\gcd(\det(M), p)=1$ such that $\bigcup_{j=1}^{(p^2-1)p^{{\bar{\eta}}-1}}{M^*}^j\mathcal {A}_p({\bar{\eta}})=\mathring{E}_{p^{\bar{\eta}}}^2$, where $\mathring{E}_{p^{\bar{\eta}}}^2$,~$\mathcal {A}_p({\bar{\eta}})$ are defined by \eqref{eq(1.2)} and \eqref{eq(1.3)}, respectively.
For simplicity, in the later of this paper,  we let $I\in M_2(\mathbb{Z})$ denote the identity matrix.

\begin{thm} \label{th(3.1)}
For a prime $p\geq 3$ and an integer $s\geq 1$, let $B=p\begin{bmatrix}
l_{1} & l_{2}\\
l_{3} & l_{4}
\end{bmatrix}+I$  be an integer matrix with $l_{i_0}\notin p\mathbb{Z}$ for some $1\leq i_0\leq 4$. Suppose $e$ is the least integer such that
$B^e=p\begin{bmatrix}
\tilde{l}_{1} & \tilde{l}_{2}\\
\tilde{l}_{3} & \tilde{l}_{4}
\end{bmatrix}+I$ satisfies $\tilde{l}_{i_0}\in p^{s-1}\mathbb{Z}$, then $e=p^{s-1}$.  Meanwhile,
$B^{p^{s-1}}=p^s\begin{bmatrix}
\check{l}_{1} & \check{l}_{2}\\
\check{l}_{3} & \check{l}_{4}
\end{bmatrix}+I$  with  $\check{l}_{i}=l_i\;({\rm mod}\ p)$ for all $1\leq i\leq 4$.
\end{thm}
\begin{proof}
 For any integer $m\geq 2$, write $B=pA+I$, we have
\begin{align}\label{eq(3.1)}
 B^{m}=(pA+I)^m=p^2\widetilde{A}_m+mpA+I,
\end{align}
where
\begin{align}\label{eq(3.2)}
\widetilde{A}_m=p^{m-2}A^m+C_p^1p^{m-3}A^{m-1}+\cdots+C_p^{p-3}pA^3+C_p^{p-2}A^2.
\end{align}

Now we prove the theorem by induction. Without loss of generality, we assume $l_2\notin p\mathbb{Z}$.

It is easy to see that the theorem holds for $s=1$. We then consider $s=2$. Since $p$ is a prime and $l_{2}\notin p\mathbb{Z}$, we infer from \eqref{eq(3.1)} that  $m=p$ is the least integer such that
$B^m=p\begin{bmatrix}
\tilde{l}^{(2)}_{1} & \tilde{l}^{(2)}_{2}\\
\tilde{l}^{(2)}_{3} & \tilde{l}^{(2)}_{4}
\end{bmatrix}+I$ satisfies $\tilde{l}^{(2)}_{2}\in p\mathbb{Z}$. Obviously,
\begin{align*}
 B^{p}=p^2\widetilde{A}_p+p^2A+I=p^2\begin{bmatrix}
\check{l}^{(2)}_{1} & \check{l}^{(2)}_{2}\\
\check{l}^{(2)}_{3} & \check{l}^{(2)}_{4}
\end{bmatrix}+I.
\end{align*}
In the following,  we prove $\check{l}^{(2)}_{i}=l_i\;({\rm mod}\ p)  \ (1\leq i\leq 4)$.
Note that $C_p^{p-2}=C_p^2=\frac{p(p-1)}{2}\in p\mathbb{Z}$ for any prime $p\geq 3$, we conclude from \eqref{eq(3.2)} that there exists an integer matrix $\widetilde{B}_p$ such that $\widetilde{A}_p=p\widetilde{B}_p$. Hence $B^{p}=p^2(p\widetilde{B}_p+A)+I$, and therefore $\check{l}^{(2)}_{i}=l_i\;({\rm mod}\ p) \ (1\leq i\leq 4)$. This proves the theorem for $s=2$.

 Inductively, we assume that the theorem holds for $s=k$. That is, $e=p^{k-1}$ is the least integer such that
$B^{e}=p\begin{bmatrix}
\tilde{l}^{(k)}_{1} & \tilde{l}^{(k)}_{2}\\
\tilde{l}^{(k)}_{3} & \tilde{l}^{(k)}_{4}
\end{bmatrix}+I$ satisfies $\tilde{l}^{(k)}_{2}\in p^{k-1}\mathbb{Z}$, moreover,
$$B^{p^{k-1}}=p^k\begin{bmatrix}
\check{l}^{(k)}_{1} & \check{l}^{(k)}_{2}\\
\check{l}^{(k)}_{3} & \check{l}^{(k)}_{4}
\end{bmatrix}+I:=p^kA_k+I
$$  with $\check{l}^{(k)}_{i}=l_i\;({\rm mod}\ p)  \ (1\leq i\leq 4)$ .

For $s=k+1$, by inductive hypothesis and the same discussion as $s=2$, we can easily show that
$$B^{p^{k}}=(B^{p^{k-1}})^p=(p^kA_k+I)^p=p^{k+1}\begin{bmatrix}
\check{l}^{(k+1)}_{1} & \check{l}^{(k+1)}_{2}\\
\check{l}^{(k+1)}_{3} & \check{l}^{(k+1)}_{4}
\end{bmatrix}+I
$$
with $\check{l}^{(k+1)}_{i}=\check{l}^{(k)}_{i}=l_i\;({\rm mod}\ p)  \ (1\leq i\leq 4)$.
We now prove that $e=p^k$ is the least integer such that
$B^{e}=p\begin{bmatrix}
\tilde{l}^{(k+1)}_{1} & \tilde{l}^{(k+1)}_{2}\\
\tilde{l}^{(k+1)}_{3} & \tilde{l}^{(k+1)}_{4}
\end{bmatrix}+I$ satisfies $\tilde{l}^{(k+1)}_{2}\in p^{k}\mathbb{Z}$.
Suppose that $n< p^k$ is the least integer which satisfies the above. By the assumption $\check{l}^{(k)}_{2}\notin p\mathbb{Z}$ for $s=k$,  we have $n> p^{k-1}$ and rewrite $n=\tau p^{k-1}+r$, where $1<\tau< p$ and $0\leq r< p^{k-1}$. It is easy to see that there exist integers $\ddot{l}_{1}$, $\ddot{l}_{2}$, $\ddot{l}_{3}$ and $\ddot{l}_{4}$ such that
$$B^{\tau p^{k-1}}=\left(p^kA_k+I\right)^\tau=p^{\tau k}A^\tau_k+\cdots+C^{2}_\tau p^{2k}A^{2}_k+C^{1}_\tau p^{k}A_k+I=p^k\begin{bmatrix}
\ddot{l}_{1} & \ddot{l}_{2}\\
\ddot{l}_{3} & \ddot{l}_{4}
\end{bmatrix}+I
$$ with $\ddot{l}_{i}=\tau\check{l}^{(k)}_{i}\;({\rm mod}\ p)  \ (1\leq i\leq 4)$. Especially $\ddot{l}_{2}\notin p\mathbb{Z}$, because $\check{l}^{(k)}_{2}\notin p\mathbb{Z}$ and $1<\tau< p$.
By using \eqref{eq(3.1)}, we can denote
$B^{r}=p\begin{bmatrix}
\tilde{m}_1 & \tilde{m}_2 \\
\tilde{m}_3 & \tilde{m}_4
\end{bmatrix}+I$ for some integers $\tilde{m}_1$, $\tilde{m}_2$, $\tilde{m}_3$ and $\tilde{m}_4$. Hence
\begin{eqnarray*} \nonumber
B^{n}&=&B^{\tau p^{k-1}+r}\\  \nonumber
&=&\begin{bmatrix}
p^{k+1}(\tilde{m}_1\ddot{l}_{1}+\tilde{m}_3\ddot{l}_{2})+p^k\ddot{l}_{1}+p\tilde{m}_1+1 & p^{k+1}(\tilde{m}_2\ddot{l}_{1}+\tilde{m}_4\ddot{l}_{2})+p^k\ddot{l}_{2}+p\tilde{m}_2 \\
p^{k+1}(\tilde{m}_1\ddot{l}_{3}+\tilde{m}_3\ddot{l}_{4})+p^k\ddot{l}_{3}+p\tilde{m}_3 & p^{k+1}(\tilde{m}_2\ddot{l}_{3}+\tilde{m}_4\ddot{l}_{4})+p^k\ddot{l}_{4}+p\tilde{m}_4+1
\end{bmatrix}\\ \nonumber
&=&p\begin{bmatrix}
\tilde{l}^{(k+1)}_{1} & \tilde{l}^{(k+1)}_{2}\\
\tilde{l}^{(k+1)}_{3} & \tilde{l}^{(k+1)}_{4}
\end{bmatrix}+I,
\end{eqnarray*}
where the above equation implies $\tilde{m}_{2}\in p^{k-1}\mathbb{Z}$ since $\tilde{l}^{(k+1)}_{2}\in p^{k}\mathbb{Z}$. So $r\geq p^{k-1}$ by the inductive hypothesis for $s=k$, this contradicts with $r< p^{k-1}$.  Therefore,  the theorem  holds for $s=k+1$. This completes the proof of Theorem \ref{th(3.1)}.
\end{proof}

\begin{re}\label{th(3.2)}
Theorem \ref{th(3.1)} may not hold if $p=2$. For example, let $B=2\begin{bmatrix}
l_1 & l_2\\
l_3 & l_4
\end{bmatrix}+I=2\begin{bmatrix}
1 & 1\\
1 & 0
\end{bmatrix}+I$, then $l_2\notin 2\mathbb{Z}$ and $B^2=2\begin{bmatrix}
2\cdot3 & 2^2\cdot1\\
 2^2\cdot1 & 2\cdot1
\end{bmatrix}+I$. This shows that $k=2^{3-1}$ is not the least integer such that
$B^k=2\begin{bmatrix}
\tilde{l}_{1} & \tilde{l}_{2}\\
\tilde{l}_{3} & \tilde{l}_{4}
\end{bmatrix}+I$ with $\tilde{l}_{2}\in 2^{3-1}\mathbb{Z}$. However, we can add some  restriction on $l_i$, such as  $l_{2}\notin 2\mathbb{Z}$ and $(l_1+1_4)\in2\mathbb{Z}$,  and get a similar result as Theorem \ref{th(3.1)} for $p=2$. We do not prove it, because it will not be used in the proof of Theorem \ref{th(1.1)} and can be similarly proved as we did in the cases  $p\geq 3$.
\end{re}

\begin{thm} \label{th(3.3)}
For a prime $p\geq3$ and an integer $s\geq 1$, let $M\in GL(2,\mathbb{F}_p)$. If $O_p(M)=\iota$ and  $M^\iota=p\begin{bmatrix}
l_{1} & l_{2}\\
l_{3} & l_{4}
\end{bmatrix}+I$. Let $k$ be the least integer such that $M^k= I\;({\rm mod}~M_2(p^s\mathbb{Z}))$, then $k\leq\iota p^{s-1}$, i.e., $O_{p^s}(M)\leq\iota p^{s-1}$. Furthermore, if there exist $l_{i_0}\notin p\mathbb{Z}$ for some $1\leq i_0\leq4$, then $k=\iota p^{s-1}$.
\end{thm}
\begin{proof}
Write $M^\iota=pA+I$, it follows from Theorem \ref{th(3.1)} that $M^{\iota p^{s-1}}=(pA+I)^{p^{s-1}}=I\;({\rm mod}~M_2(p^s\mathbb{Z}))$. Thus $k\leq\iota p^{s-1}$. In particular, we claim that there exists an integer $k'$ such that $k=\iota k'$. If otherwise, $k=\iota\tilde{k}+r$ for some integers $\tilde{k}$ and  $1\leq r\leq \iota-1$. Note that $O_p(M)=\iota$, we have $M^{\iota\tilde{k}}=I\;({\rm mod}~M_2(p\mathbb{Z}))$. Consequently, $I=M^k=M^{\iota\tilde{k}+r}=IM^r=M^r\;({\rm mod}~M_2(p\mathbb{Z}))$, which contradicts with $O_p(M)=\iota>r$. Hence $k=\iota k'$.

Next, we prove $k=\iota p^{s-1}$ if there exist  $l_{i_0}\notin p\mathbb{Z}$ for some $1\leq i_0\leq 4$. Due to $k=\iota k'=O_{p^s}(M)$, then there exist integers $l_{1}'$, $l_{2}'$, $l_{3}'$ and $l_{4}'$ such that
\begin{equation}\label{eq(3.3)}
M^k=M^{\iota k'}=\left(p\begin{bmatrix}
l_{1} & l_{2}\\
l_{3} & l_{4}
\end{bmatrix}+I\right)^{k'}=p^s\begin{bmatrix}
l_{1}' & l_{2}'\\
l_{3}' & l_{4}'
\end{bmatrix}+I.
\end{equation}
Since $l_{i_0}\notin p\mathbb{Z}$, we deduce from  (\ref{eq(3.3)}) and Theorem \ref{th(3.1)}  that $k'\geq p^{s-1}$. Hence $k\geq \iota p^{s-1}$, together with $k\leq\iota p^{s-1}$, shows that $k=\iota p^{s-1}$.
\end{proof}

Let $M\in GL(2,\mathbb{F}_p)$ with $O_{p}(M)=\iota$ and let $\xi\in \mathring{E}_{p^s}^2$.  According to Theorem \ref{th(3.3)}, we have $\bigcup_{j=1}^{\infty} M^j\xi=\bigcup_{j=1}^{\iota\cdot p^{s-1}} M^j\xi\;({\rm mod}~\mathbb{Z}^2)$. Assume integers $p\geq 2,s\geq 1$, let
\begin{equation}\label{eq(3.4)}
\mathcal{T}_{p,s}=\{l: 0\leq l\leq p^s-1, l\in\mathbb{Z}\setminus p\mathbb{Z}\}.
\end{equation}
For any $l\in \mathcal{T}_{p,s}$, define
\begin{equation}\label{eq(3.5)}
 \mathcal{B}_{p,s}(l)=\bigcup_{j=1}^{\iota p^{s-1}} M^j \begin{pmatrix}
 \frac{p^{s-1}}{p^{s}}\\ \frac{l}{p^{s}}
  \end{pmatrix}\;({\rm mod}~\mathbb{Z}^2)
\end{equation}
and
\begin{equation}\label{eq(3.6)}
Q_{p,s}(l)=\left\{l':\mathcal{B}_{p,s}(l)=\mathcal{B}_{p,s}(l'), l'\in \mathcal{T}_{p,s}  \right\}.
\end{equation}
It is clear that if $l'\in Q_{p,s}(l)$, then $Q_{p,s}(l')=Q_{p,s}(l)$. Hence there exists a  nonnegative integer $m_{p,s}$ such that the set $\mathcal{T}_{p,s}$ can be divided into disjoint union $ \mathcal{T}_{p,s}=\bigcup_{i=0}^{m_{p,s}}Q_{p,s}(l_i)$. Note that $O_{p^s}(M)\leq\iota p^{s-1}$, we
can easily show that for any $l_1, l_2\in \mathcal{T}_{p,s}$, either $\mathcal{B}_{p,s}(l_1)=\mathcal{B}_{p,s}(l_2)$ or $\mathcal{B}_{p,s}(l_1)\bigcap\mathcal{B}_{p,s}(l_2)=\emptyset$.

Let $|E|$ denote the cardinality of set $E$, the following theorem describes the properties of $\mathcal{B}_{p,s}(l)$ and $Q_{p,s}(l)$ for  $l\in  \mathcal{T}_{p,s}$.

\begin{thm} \label{th(3.4)}
For a prime $p\geq3$ and an integer $s\geq 1$,  let $\mathring{E}_{p^s}^2$, $\mathcal {A}_p(s)$, $\mathcal{T}_{p,s}$, $\mathcal{B}_{p,s}$, $Q_{p,s}$ be defined by \eqref{eq(1.2)}, \eqref{eq(1.3)}, (\ref{eq(3.4)}), (\ref{eq(3.5)}) and (\ref{eq(3.6)}), respectively. If $O_{p}(M)=p^2-1$ and $M^{p^2-1}= p\begin{bmatrix}
l_{1} & l_{2}\\
l_{3} & l_{4}
\end{bmatrix}+I$ with $l_2\notin p\mathbb{Z}$,   then

(i) $\mathcal{B}_{p,s}(l)$ has $(p^2-1) p^{s-1}$ different elements for any $l\in \mathcal{T}_{p,s}$, i.e., $|\mathcal{B}_{p,s}(l)|=(p^2-1) p^{s-1}$;

(ii) $|Q_{p,s}(l)|= p-1$ for any $l\in \mathcal{T}_{p,s}$;

(iii) For any integer ${\bar{\eta}}\geq 1$,
\begin{equation*}
\bigcup_{j=1}^{(p^2-1)p^{{\bar{\eta}}-1}} M^j\mathcal {A}_p({\bar{\eta}}  )=\bigcup_{s=1}^{{\bar{\eta}}}\bigcup_{l\in \mathcal{T}_{p,s}}\mathcal{B}_{p,s}(l)=\mathring{E}_{p^{{\bar{\eta}} }}^2\;({\rm mod}\; \Bbb Z^2).
\end{equation*}

\end{thm}
\begin{proof}
We first prove the following claim.
\vspace{0.3cm}\\
\emph{{\bf Claim 1.} For any $l,l'\in \mathcal{T}_{p,s}$ with $l=l'\;({\rm mod}~p)$, if there exists an integer $1\leq k\leq (p^2-1)p^{s-1}$ such that
\begin{equation}\label{eq(3.7)}
M^{k} \frac{1}{p^s}\begin{pmatrix}
 p^{s-1}\\ l
  \end{pmatrix}=\frac{1}{p^s}\begin{pmatrix}
 p^{s-1}\\ l'
  \end{pmatrix}\;({\rm mod}~\mathbb{Z}^2),
\end{equation}
then $k=(p^2-1)p^{s-1}$ and $l=l'$.}
\begin{proof}[Proof of Claim 1]
Since $v:=(p^{s-1},l)^t=(p^{s-1},l')^t:=v'\;({\rm mod}~p\mathbb{Z}^2)$ and $O_p(M)=p^2-1$, it can be easily proved that $k=(p^2-1)k'$ for some integer $1\le k'\leq p^{s-1}$.
Otherwise $k=(p^2-1)k''+r$ for some integers $k''$ and $1\leq r< p^2-1$. It follows from $O_p(M)=p^2-1$ and Lemma \ref{th(2.12)} that $M$ is an ergodic matrix and $M^kv=M^{(p^2-1)k''+r}v=M^{r}v=v'=v\;({\rm mod}\ p\mathbb{Z}^2)$. Therefore, $M^{r+1}v=Mv\;({\rm mod}~p\mathbb{Z}^2)$ and
$$
|\{Mv,\cdots,M^rv,M^{r+1}v,\cdots,M^{p^2-1}v\}\;({\rm mod} \ p\mathbb{Z}^2)|<p^2-1,
$$
which contradicts the ergodicity of $M$. Hence $k=(p^2-1)k'$.

Next we prove $k'=p^{s-1}$. Denote $M^{p^2-1}=pA+I$, by \eqref{eq(3.1)} and \eqref{eq(3.7)}, there exist integers $\tilde{l}_{1},\tilde{l}_{2}, \tilde{l}_{3}$ and $\tilde{l}_{4}$ such that
\begin{eqnarray*} \nonumber
M^{k} \begin{pmatrix}
 p^{s-1}\\ l
  \end{pmatrix}
&=&(pA+I)^{k'}\begin{pmatrix}
 p^{s-1}\\ l
  \end{pmatrix}
=\begin{bmatrix}
p\tilde{l}_{1}+1 & p\tilde{l}_{2}\\
p\tilde{l}_{3} & p\tilde{l}_{4}+1
\end{bmatrix}\begin{pmatrix}
 p^{s-1}\\ l
  \end{pmatrix}\\
 &=& \begin{pmatrix}
 p^s\tilde{l}_{1}+p^{s-1}+p\tilde{l}_{2}l\\ p^s\tilde{l}_{3}+p\tilde{l}_{4}l+l
  \end{pmatrix}=\begin{pmatrix}
 p^{s-1}\\ l'
  \end{pmatrix}\;({\rm mod}~p^s\mathbb{Z}^2).
\end{eqnarray*}
This together with $l\notin p\mathbb{Z}$ yields $\tilde{l}_{2}=p^{s-1} m$ for some integer $m$. As $l_2\notin p\mathbb{Z}$, Theorem \ref{th(3.1)} implies $k'\geq p^{s-1}$. Combining  $k'\leq p^{s-1}$ shows that $k'=p^{s-1}$, hence $k=(p^2-1)p^{s-1}$. According to $l_2\notin p\mathbb{Z}$ and Theorem \ref{th(3.3)}, we have $O_{p^s}(M)=(p^2-1) p^{s-1}$, i.e., $M^{(p^2-1) p^{s-1}}=I\;({\rm mod}\; M_2(p^s\mathbb{Z}))$. Hence $l=l'$, which completes the proof of the claim.
\end{proof}
We now continue with the proof of Theorem \ref{th(3.4)}.

(i) Suppose $|\mathcal{B}_{p,s}(l)|<(p^2-1) p^{s-1}$, then there exist
  $1\leq k_1<k_2\leq(p^2-1) p^{s-1}$ such that  $M^{k_1} \begin{pmatrix}
 \frac{p^{s-1}}{p^{s}}\\ \frac{l}{p^{s}}
  \end{pmatrix}=M^{k_2} \begin{pmatrix}
 \frac{p^{s-1}}{p^{s}}\\ \frac{l}{p^{s}}
  \end{pmatrix}\pmod{\mathbb{Z}^2}$. Note that $O_{p^s}(M)=(p^2-1) p^{s-1}$, multiplying $M^{(p^2-1) p^{s-1}-k_2}$ on both sides of the above equation, we get
\begin{equation}\label{eq(3.8)}
M^{k_1+(p^2-1) p^{s-1}-k_2} \begin{pmatrix}
 \frac{p^{s-1}}{p^{s}}\\ \frac{l}{p^{s}}
  \end{pmatrix}=M^{k_2+(p^2-1) p^{s-1}-k_2} \begin{pmatrix}
 \frac{p^{s-1}}{p^{s}}\\ \frac{l}{p^{s}}
  \end{pmatrix}=\begin{pmatrix}
 \frac{p^{s-1}}{p^{s}}\\ \frac{l}{p^{s}}
  \end{pmatrix}\;({\rm mod}~\mathbb{Z}^2).
 \end{equation}
However, Claim 1 shows that \eqref{eq(3.8)} does not hold because $k_1+(p^2-1) p^{s-1}-k_2<(p^2-1) p^{s-1}$. Hence $\mathcal{B}_{p,s}(l)$ has $(p^2-1) p^{s-1}$ different elements for any $l\in \mathcal{T}_{p,s}$.

(ii) We first prove  $|Q_{p,s}(l)|\leq p-1$. Assume that there exists $l\in \mathcal{T}_{p,s}$ such that $|Q_{p,s}(l)|\geq p$, then there exist $l',l''\in Q_{p,s}(l)$ satisfies $l'\neq l''$ and $l'=l''\;({\rm mod}~p)$. Combining $O_{p^s}(M)=(p^2-1)p^{s-1}$ and $\mathcal{B}_{p,s}(l')=\mathcal{B}_{p,s}(l'')$, we deduce that there exists a
positive integer $k<(p^2-1)p^{s-1}$ such that
\begin{equation*}
M^{k} \frac{1}{p^s}\begin{pmatrix}
 p^{s-1}\\ l'
  \end{pmatrix}=M^{(p^2-1)p^{s-1}} \frac{1}{p^s}\begin{pmatrix}
 p^{s-1}\\ l''
  \end{pmatrix}=\frac{1}{p^s}\begin{pmatrix}
 p^{s-1}\\ l''
  \end{pmatrix}\;({\rm mod}~\mathbb{Z}^2),
\end{equation*}
which  contradicts with Claim 1. Hence $|Q_{p,s}(l)|\leq p-1$.

We now prove $|Q_{p,s}(l)|\ge p-1$. Suppose on the contrary that there exists $l_0 \in \mathcal{T}_{p,s}$ such that $|Q_{p,s}(l_0)|\leq p-2$, we decompose the set $\mathcal{T}_{p,s}$ into disjoint union
$Q_{p,s}(l_0)\bigcup\bigcup_{i=1}^{m_{p,s}}Q_{p,s}({l_i})$. Note that $|Q_{p,s}(l_0)|\leq p-2$ and $|Q_{p,s}({l_i})|\leq p-1$ for all $1\leq i\leq m_{p,s}$, we have
$$
p^s-p^{s-1}=|\mathcal{T}_{p,s}|=\sum_{i=0}^{m_{p,s}}|Q_{p,s}({l_i})|\leq p-2+{m_{p,s}}(p-1),
$$
this shows that ${m_s}\geq \frac{p^s-p^{s-1}-p+2}{p-1}$.

From the definition of $Q_{p,s}$, we see that $|\bigcup_{l\in\mathcal{T}_{p,s}}  \mathcal{B}_{p,s}(l)| =|\bigcup_{i=0}^{m_{p,s}}\bigcup_{l\in Q_{p,s}({l_i})} \mathcal{B}_{p,s}(l)|=|\bigcup_{i=0}^{m_{p,s}}\mathcal{B}_{p,s}(l_i)|$. Note that $|\mathcal{B}_{p,s}(l)|=(p^2-1) p^{s-1}$ for any $l\in \mathcal{T}_{p,s}$ and $\mathcal{B}_{p,s}(l_i)\cap \mathcal{B}_{p,s}(l_j)=\emptyset$ for any $0\leq i\neq j\leq {m_{p,s}}$, we obtain
\begin{align*}
\left|\bigcup_{i=0}^{m_{p,s}} \mathcal{B}_{p,s}(l_i)\right|=({m_{p,s}}+1)(p^2-1) p^{s-1} &\geq(\frac{p^s-p^{s-1}-p+2}{p-1}+1)(p^2-1) p^{s-1}\\
&= p^{2s}-p^{2(s-1)}+p^s+p^{s-1}> p^{2s}-p^{2(s-1)}.
\end{align*}
However, by Remark \ref{th(2.6)}, we have $\bigcup_{l\in \mathcal{T}_{p,s}}\mathcal{B}_{p,s}(l)\subset \mathring{E}_{p^{s}}^2\setminus \mathring{E}_{p^{s-1}}^2$. Hence
$$\left|\bigcup_{i=0}^{m_{p,s}} \mathcal{B}_{p,s}(l_i)\right|=\left| \bigcup_{l\in \mathcal{T}_{p,s}}\mathcal{B}_{p,s}(l)\right| \le \left|\mathring{E}_{p^{s}}^2\setminus \mathring{E}_{p^{s-1}}^2\right|=p^{2s}-p^{2(s-1)}.$$
This contradiction yields $|Q_{p,s}(l)|\geq p-1$. Therefore, $|Q_{p,s}(l)|=p-1$ for any $l\in \mathcal{T}_{p,s}$.

(iii) From the definition of $\mathcal {A}_p(\bar{\eta})$ and $\mathcal{B}_{p,s}$, the first equation is clearly established. Next we prove  $\bigcup_{s=1}^{{\bar{\eta}}}\bigcup_{l\in \mathcal{T}_{p,s}}\mathcal{B}_{p,s}(l)=\mathring{E}_{p^{{\bar{\eta}} }}^2\;({\rm mod}\; \Bbb Z^2)$.
Since $|Q_{p,s}(l)|=p-1$, the set $\mathcal{T}_{p,s}$ can be decomposed into disjoint union
$\bigcup_{i=1}^{p^{s-1}}Q_{p,s}({l_i})$. Note that $\mathcal{B}_{p,s}(l_i)\cap \mathcal{B}_{p,s}(l_j)=\emptyset$ for any $1\leq i\neq j\leq p^{s-1}$ and $|\mathcal{B}_{p,s}(l)|=(p^2-1) p^{s-1}$ for any $l\in\mathcal{T}_{p,s}$, we get
$$
\left|\bigcup_{l\in \mathcal{T}_{p,s}}\mathcal{B}_{p,s}(l)\right|=\left|\bigcup_{i=1}^{p^{s-1}}\mathcal{B}_{p,s}(l_i)\right|=p^{s-1}(p^2-1)p^{s-1}=p^{2s}-p^{2s-2}.
$$
For any $s\neq s'$, it is easy to see that $\left(\bigcup_{l\in \mathcal{T}_{p,s}}\mathcal{B}_{p,s}(l)\right)\bigcap \left(\bigcup_{l\in \mathcal{T}_{p,s'}}\mathcal{B}_{p,s'}(l)\right)=\emptyset$  by  Remark \ref{th(2.6)}. Hence
\begin{equation}\label{eq(3.9)}
\left|\bigcup_{s=1}^{{\bar{\eta}}}\bigcup_{l\in \mathcal{T}_{p,s}}\mathcal{B}_{p,s}(l)\right|=\sum_{s=1}^{{\bar{\eta}} }(p^{2s}-p^{2(s-1)})=p^{2{\bar{\eta}} }-1=|\mathring{E}_{p^{\bar{\eta}}}^2|.
\end{equation}
It follows from Lemma \ref{th(2.5)}
that $\bigcup_{s=1}^{\bar{\eta}}\bigcup_{l\in \mathcal{T}_{p,s}}\mathcal{B}_{p,s}(l)\subset\mathring{E}_{p^{\bar{\eta}}}^2\;({\rm mod}\; \Bbb Z^2)$.
Combining this with \eqref{eq(3.9)}, we obtain $\bigcup_{s=1}^{\bar{\eta}}\bigcup_{l\in \mathcal{T}_{p,s}}\mathcal{B}_{p,s}(l)=\mathring{E}_{p^{\bar{\eta}}}^2\;({\rm mod}\; \Bbb Z^2).$
\end{proof}

\begin{re}\label{th(3.5)}
For the matrix $M$ satisfies the conditions of Theorem \ref{th(3.4)}, by Claim 1, we have  $\mathcal{B}_{p,s}(l)\bigcap \mathcal{B}_{p,s}(l')=\emptyset$ for any $l,l'\in \mathcal{T}_{p,s}$ with $ l\neq l'$  and $l=l'\;({\rm mod} \ p)$. Let $1\leq i\leq p-1$, we define $\mathcal{R}_{p,s}(i)=\{l:l=i\;({\rm mod} \ p),0\leq l\leq p^s-1, l\in\mathbb{Z}\}$, then $|\mathcal{R}_{p,s}(i)|=p^{s-1}$. Hence
Theorem \ref{th(3.4)}(i) implies that
\begin{equation*}
\left|\bigcup_{l\in \mathcal{R}_{p,s}(i)}\mathcal{B}_{p,s}(l)\right|=p^{s-1}(p^2-1) p^{s-1}=p^{2s}-p^{2(s-2)}
\end{equation*}
for any $1\leq i\leq p-1$. By Remark \ref{th(2.6)}, we have
\begin{equation}\label{eq(3.10)}
\bigcup_{l\in \mathcal{R}_{p,s}(i)}\mathcal{B}_{p,s}(l)=\mathring{E}_{p^{s}}^2\setminus \mathring{E}_{p^{s-1}}^2\;({\rm mod} \ \Bbb Z^2).
\end{equation}
\end{re}
 Let $\det (M)=L$ and $\varphi(q)$ be the Euler's phi function. If  $\gcd(L,q)=1$, it follows from Lemma \ref{th(2.7)} that there exists an integer $n$ such that
\begin{align}\label{eq(3.11)}
L^{\varphi(q)}=nq+1.
\end{align}
To prove Theorem \ref{th(1.2)}, we  need the following lemma, which was proved in \cite{Liu-Dong-Li_2017}.
\begin{lemma} \label{th(3.6)}
Let $M\in M_2(\mathbb{Z})$ be an expanding matrix with $\det(M)=L$ and $\mathring{E}_q^2$ be defined by \eqref{eq(1.2)}. If $\gcd(L,q)=1$, then
\begin{align*}
M^{*j}(\lambda+\Bbb Z^2)\supset L^{\varphi(q)j}(M^{*j}\lambda+\Bbb Z^2)~~ \hbox { for all } \lambda \in \mathring{E}_{q}^2.
\end{align*}
\end{lemma}

\begin{proof}[Proof of Theorem \ref{th(1.2)}] First, we prove that $n^*(\mu_{M,D})\leq p^{2{\bar{\eta}}}$.  Since $M$ is an integer matrix and $\gcd(\det(M),p)=1$, by Lemma \ref{th(2.5)}, we have $M^*(
\mathring{E}_{p^{\bar{\eta}}}^2+\mathbb{Z}^2)=M^*(
\mathring{E}_{p^{\bar{\eta}}}^2)+M^*(\mathbb{Z}^2)\subset M^*(
\mathring{E}_{p^{\bar{\eta}}}^2)+\mathbb{Z}^2=\mathring{E}_{p^{\bar{\eta}} }^2+\mathbb{Z}^2$. It follows from $\mathcal{ Z}(m_D)\subset \mathring{E}_{p^{\bar{\eta}}}^2+\mathbb{Z}^2$ and Lemma \ref{th(2.4)} that
\begin{align}\label{eq(3.12)}
n^*(\mu_{M,D})\leq |\mathring{E}_{p^{\bar{\eta}}}^2|+1=p^{2{\bar{\eta}} }.
\end{align}

Second, we will show that there exists an expanding integer matrix $M_0$ such that  $n^*(\mu_{M_0,D})= p^{2{\bar{\eta}}}$ if $p$ is a prime. Assume that there exists an expanding matrix $M_0\in GL(2,\mathbb{F}_p)$ such that $M_0^*$ satisfies Theorem \ref{th(3.4)} (we will prove its existence at the end of the proof). Let $\det(M_0)=L$, then $\gcd(L,p)=1$. From Theorems \ref{th(3.3)} and \ref{th(3.4)}, we have $O_p(M_0^*)=p^2-1$, $O_{p^{\bar{\eta}}}(M_0^*)=(p^2-1) p^{{\bar{\eta}}-1}$ and
  \begin{equation}\label{eq(3.13)}
\bigcup_{j=1}^{(p^2-1)p^{{\bar{\eta}}-1}}{M_0^*}^j\mathcal {A}_p({\bar{\eta}})=\mathring{E}_{p^{\bar{\eta}}}^2\;({\rm{mod}}  \ \mathbb{Z}^2).
\end{equation}
Since $\mathcal{N}(\mathcal {A}_p({\bar{\eta}})+\mathbb{Z}^2)\subset \mathcal{ Z}(m_D)$ and $O_{p^{\bar{\eta}}}(M_0^*)=(p^2-1) p^{{\bar{\eta}}-1}$, by \eqref{eq(2.3)} and Lemma \ref{th(3.6)}, we have
\begin{align}\label{eq(3.14)} \nonumber
\mathcal{ Z}(\hat{\mu}_{M_0,D})&=\bigcup_{j=1}^{\infty}M_0^{*j}(\mathcal{ Z}(m_{D}))\supset \mathcal{N}\bigcup_{j=1}^{\infty}M_0^{*j}(\mathcal {A}_p({\bar{\eta}})+\Bbb Z^2)\\
&\supset \mathcal{N}\bigcup_{j=1}^{\infty}L^{\varphi(p^{\bar{\eta}} )j}(M_0^{*j}\mathcal {A}_p({\bar{\eta}})+\Bbb Z^2)\supset \mathcal{N}\bigcup_{j=1}^{(p^2-1) p^{{\bar{\eta}}-1}}L^{\varphi(p^{\bar{\eta}})j}(M_0^{*j}\mathcal {A}_p({\bar{\eta}})+\Bbb Z^2).
\end{align}
Let $\Lambda=\mathcal{N}L^{\varphi(p^{\bar{\eta}})(p^2-1) p^{{\bar{\eta}}-1}} {E}_{p^{\bar{\eta}}}^2 $, we will show that $\mathcal {E}_\Lambda=\{e^{2\pi i\langle\lambda,x\rangle}:\lambda\in\Lambda\}$ is an orthogonal set of $L^2(\mu_{M,D})$. For any $\lambda_1\ne \lambda_2\in\Lambda$, there exists $\lambda^\prime\in\mathring{E}_{p^{\bar{\eta}}}^2 \;({\rm mod}\; \Bbb Z^2)$ such that $\lambda_1-\lambda_2= \mathcal{N}L^{\varphi(p^{\bar{\eta}})(p^2-1) p^{{\bar{\eta}} -1}}\lambda^\prime$. By~$(\ref{eq(3.13)})$, there exist $\lambda_0\in \mathcal {A}_p({\bar{\eta}})$ and $1\leq j_0\leq (p^2-1) p^{{\bar{\eta}}-1}$ such that $\lambda^\prime={M_0^*}^{j_0}\lambda_0\;({\rm mod}\; \Bbb Z^2)$. Then
\begin{align}\label{eq(3.15)}\nonumber
\lambda_1-\lambda_2 &\in \mathcal{N}L^{\varphi(p^{\bar{\eta}})(p^2-1) p^{{\bar{\eta}} -1}}({M_0^*}^{j_0}\lambda_0+\Bbb Z^2)\\
 &=\mathcal{N}L^{\varphi(p^{\bar{\eta}}){j_0}}(L^{\varphi(p^{\bar{\eta}} )((p^2-1) p^{{\bar{\eta}}-1}-{j_0})}({M_0^*}^{j_0}\lambda_0+\Bbb Z^2)).
\end{align}
Since $\gcd(L,p^{{\bar{\eta}}})=1$, by using (\ref{eq(3.11)}), we obtain $L^{\varphi(p^{\bar{\eta}})((p^2-1) p^{{\bar{\eta}} -1}-{j_0})}=p^{{\bar{\eta}}}m^\prime+1$ for some integer $m^\prime$. It follows from $p^{{\bar{\eta}}}m^\prime {M_0^*}^{j_0}\lambda_0\in \Bbb Z^2$ that
$L^{\varphi(p^{\bar{\eta}})((p^2-1) p^{{\bar{\eta}} -1}-{j_0})}({M_0^*}^{j_0}\lambda_0+\Bbb Z^2)=(p^{{\bar{\eta}} }m^\prime+1)({M_0^*}^{j_0}\lambda_0+\Bbb Z^2)\subset {M_0^*}^{j_0}\lambda_0+\Bbb Z^2$.
Hence, by $(\ref{eq(3.14)})$ and $(\ref{eq(3.15)})$, we have
\begin{align*}
\lambda_1-\lambda_2 \in \mathcal{N}L^{\varphi(p^{\bar{\eta}}){j_0}} ({M_0^*}^{j_0}\lambda_0+\Bbb Z^2)\subset \mathcal{N}L^{\varphi(p^{\bar{\eta}})j_0}(M_0^{*j_0}\mathcal {A}_p({\bar{\eta}})+\Bbb Z^2) \subset\mathcal{ Z}(\hat{\mu}_{M_0,D}).
\end{align*}
This shows that $(\Lambda-\Lambda)\setminus\{0\}\subset\mathcal{ Z}(\hat{\mu}_{M_0,D})$, by~$(\ref{eq(2.2)})$,  the elements in $\mathcal {E}_\Lambda$ are mutually orthogonal.  Hence $n^*(\mu_{M_0,D})\geq p^{2{\bar{\eta}}}$, and \eqref{eq(3.12)} gives $n^*(\mu_{M_0,D})= p^{2{\bar{\eta}}}$.

Finally, we prove that there exists an expanding matrix $M_0\in GL(2,\mathbb{F}_p)$  such that $M_0^*$ satisfies  Theorem \ref{th(3.4)}. It is well known that the matrix in $ GL(2,\mathbb{F}_p)$ with order equals $p^2-1$ always exists. Let $M_0\in GL(2,\mathbb{F}_p)$ with $O_p(M_0)=p^2-1$ and ${M_0}^{p^2-1}= p\begin{bmatrix}
l_{1} & l_{2}\\
l_{3} & l_{4}
\end{bmatrix}+I$.

(a) If $l_2\notin p\mathbb{Z}$ or $l_3\notin p\mathbb{Z}$.
Let $\widetilde{M}_0=\widetilde {N}M_0$, where $\widetilde {N}=p^2n+1$ is a sufficient large integer such that $\widetilde{M}_0$ is an expanding integer matrix.
Then $O_p(\widetilde{M}_0)=O_p(M_0)=p^2-1$ and it is clear that there exists an integer $n'$ such that $(p^2n+1)^{p^2-1}=p^2n'+1$,  hence
\begin{align}\label{eq(3.16)}
\widetilde{M}_0^{p^2-1}={\widetilde {N}M_0}^{p^2-1}=p\begin{bmatrix}
p^2n'l_1+pn'+l_1 & p^2n'l_2+l_2\\
p^2n'l_3+l_3 & p^2n'l_4+pn'+l_4
\end{bmatrix}+I.
\end{align}
Let $\mathscr{M}_0=\widetilde{M}_0^*$ if $l_2\notin p\mathbb{Z}$ or  $\mathscr{M}_0=\widetilde{M}_0$ if $l_3\notin p\mathbb{Z}$. Then \eqref{eq(3.16)} shows that $\mathscr{M}_0^*$ satisfies  Theorem \ref{th(3.4)}, and therefore $n^*(\mu_{\mathscr{M}_0,D})= p^{2{\bar{\eta}}}$.

(b) If $l_2,l_3\in p\mathbb{Z}$. Let
$\widetilde{M}_1=M_0+pI$. Then $O_p(\widetilde{M}_1)=O_p(M_0)$ and
\begin{equation}\label{eq(3.17)}
(\widetilde{M}_1)^{p^2-1}=\sum_{r=2}^{p^2-1}C_{p^2-1}^rM_0^{p^2-1-r}p^r+{(p^2-1)}pM_0^{p^2-2}+M_0^{p^2-1}:=p\begin{bmatrix}
l'_{1} & l'_{2}\\
l'_{3} & l'_{4}
\end{bmatrix}+I.
\end{equation}
We will prove $l'_{2}\notin p\mathbb{Z}$ or $l'_{3}\notin p\mathbb{Z}$. Since $M_0^{p^2-1}=p\begin{bmatrix}
l_{1} & l_{2}\\
l_{3} & l_{4}
\end{bmatrix}+I$ with $l_2,l_3\in p\mathbb{Z}$, by \eqref{eq(3.17)}, we only need to show that $M_0^{p^2-2}=\begin{bmatrix}
m_1 & m_2\\
m_3 & m_4
\end{bmatrix}$ satisfies $m_2\notin p\mathbb{Z}$ or $m_3\notin p\mathbb{Z}$. Suppose on the  contrary that $m_2=ps_2, m_3=ps_3$. Then $M_0\in GL(2,\mathbb{F}_p)$ implies $m_1,m_4\notin p\mathbb{Z}$.  Let $M_0=\begin{bmatrix}
r_1 & r_2\\
r_3 & r_4
\end{bmatrix}$, we have
\begin{align}\label{eq(3.18)}
M_0^{p^2-1}=M_0^{p^2-2}\cdot M_0
=\begin{bmatrix}
ps_2r_3+m_1r_1 & ps_2r_4+m_1r_2\\
ps_3r_1+m_4r_3 & ps_3r_2+m_4r_4
\end{bmatrix}=p\begin{bmatrix}
l_{1} & l_{2}\\
l_{3} & l_{4}
\end{bmatrix}+I.
\end{align}
Since $l_2,l_3\in p\mathbb{Z}$, by $m_1,m_4\notin p\mathbb{Z}$ and \eqref{eq(3.18)}, we have $r_2,r_3 \in p\mathbb{Z}$. However, if $r_2,r_3 \in p\mathbb{Z}$, according to Lemma \ref{th(2.7)}, $M_0^{p-1}=\begin{bmatrix}
r_1 & 0\\
0 & r_4
\end{bmatrix}^{p-1}=\begin{bmatrix}
r_1^{p-1} & 0\\
0 & r_4^{p-1}
\end{bmatrix}=I\;({\rm{ mod} } \ M_2(p\mathbb{Z}))$. This contradicts with $O_p(M_0)=p^2-1$. Hence $m_2\notin p\mathbb{Z}$ or $m_3\notin p\mathbb{Z}$, which implies that $l'_{2}\notin p\mathbb{Z}$ or $l'_{3}\notin p\mathbb{Z}$. The remained proof is the same as case (a). This proves the existence.
\end{proof}

Now we are ready to prove  Theorem \ref{th(1.1)}. In order to prove it, we need the following lemma of \cite{Liu-Dong-Li_2017}.

\begin{lemma} \label{th(3.7)}{\rm \cite[Theorem 1.3]{Liu-Dong-Li_2017}} Let $M\in M_2(\mathbb{Z})$ be an expanding matrix and $D\subset\mathbb{Z}^2$ be a finite subset, and let
$\mu_{M,D}$, $\mathring{E}_{q}^2$, $\mathcal{Z}_{D}^2$, $n^*(\mu_{M,D})$ be defined by \eqref{eq(1.1)}, \eqref{eq(1.2)}, \eqref{eq(1.5)} and \eqref{eq(2.4)}, respectively. If $\emptyset\neq\mathcal{Z}_{D}^2\subset \mathring{E}_{q}^2$ and
$\gcd(\det(M),q)=1$, then
\begin{align*}
n^*(\mu_{M,D})
\left\{
\begin{array}{ll}
< q^2,\quad &{\rm if}\; \bigcup_{j=1}^{q^2-1}{M^*}^j\mathcal{Z}_{D}^2\subsetneqq \mathring{E}_q^2\;({\rm{mod}}  \ \mathbb{Z}^2), \\
=q^2,\quad &{\rm if}\; \bigcup_{j=1}^{q^2-1}{M^*}^j\mathcal{Z}_{D}^2=\mathring{E}_q^2 \;({\rm{mod}}  \ \mathbb{Z}^2).
\end{array}
\right.
\end{align*}
\end{lemma}

\begin{proof}[Proof of Theorem \ref{th(1.1)}]  (i) Suppose $2\alpha_1-\beta_1\notin 3\mathbb{Z}$ or $2\alpha_2-\beta_2\notin 3\mathbb{Z}$.
First, we prove $n^*(\mu_{M,D})\leq 9$.
Let
\begin{eqnarray*}
A_1=\begin{bmatrix}
 \alpha_1\beta_2-\alpha_2\beta_1&0 \\
 0&\alpha_1\beta_2-\alpha_2\beta_1
\end{bmatrix},
\end{eqnarray*}
then $D_1=A_1^{-1}D=\frac{1}{\alpha_1\beta_2-\alpha_2\beta_1}D$ and $A_1^{-1}MA_1=M$. By Lemma \ref{th(2.2)}, we have $n^*(\mu_{M,D_1})=n^*(\mu_{M,D})$.
It is well known that
$
1+e^{2\pi i\theta_1}+e^{2\pi i\theta_2}=0
$
if and only if
\begin{align}\label{eq(3.19)}
 \left\{
 \begin{array}{ll}
\theta_1=1/3+k_1,
\;\;\\
\theta_2=2/3+k_2,
\end{array}
\right.
\;\;\text{or}\;\;\;\;\;\
\left\{
 \begin{array}{ll}
\theta_1=2/3+k_3,
\;\;\\
\theta_2=1/3+k_4,
\end{array}
\right.
\end{align}
where~$k_1,k_2,k_3,k_4\in \mathbb{Z}$. By \eqref{eq(3.19)},  we can easily obtain
\begin{eqnarray}\label{eq(3.20)}
\mathcal{Z}(m_{D_1})=\{x\in\mathbb{R}^n:m_{D_1}(x)=0\}:=Z_0 \cup \widetilde{Z}_0,
\end{eqnarray}
where
\begin{eqnarray*}
Z_0=\left\{
\begin{pmatrix}
  \frac{\beta_2-2\alpha_2}{3} \\ \frac{2\alpha_1-\beta_1}{3}
  \end{pmatrix}\
  +
  \begin{pmatrix}
  {\beta_2 k_1-\alpha_2 k_2} \\ {\alpha_1 k_2-\beta_1 k_1}
  \end{pmatrix}\
  :k_1,k_2\in\mathbb{Z}
  \right\},
 \end{eqnarray*}
and
\begin{eqnarray*}
\widetilde{Z}_0=\left\{
\begin{pmatrix}
  \frac{2\beta_2-\alpha_2}{3} \\ \frac{\alpha_1-2\beta_1}{3}
  \end{pmatrix}\
  +
  \begin{pmatrix}
  {\beta_2\widetilde{k}_1-\alpha_2\widetilde{k}_2}\\ {\alpha_1\widetilde{k}_2-\beta_1\widetilde{k}_1}
  \end{pmatrix}\
  :\widetilde{k}_1,\widetilde{k}_2\in\mathbb{Z}
  \right\}.
 \end{eqnarray*}
Since $2\alpha_1-\beta_1\notin 3\mathbb{Z}$ or $2\alpha_2-\beta_2\notin 3\mathbb{Z}$, \eqref{eq(3.20)} yields
$\mathcal{Z}(m_{D_1})\subset \mathring{E}_3^2+\mathbb{Z}^2$.
 As $\gcd(\det(M),3)=1$, we conclude from Lemma \ref{th(2.5)} that $M^{*}( \mathring{E}_3^2+\mathbb{Z}^2)\subset  \mathring{E}_3^2+\mathbb{Z}^2$. Therefore, Lemma \ref{th(2.4)} implies $n^*(\mu_{M,D})=n^*(\mu_{M,D_1})\leq |\mathring{E}_3^2|+1=9$.

Second, we show that the number $9$ is the best. We will prove it in two cases: Case 1, $\alpha_1\beta_2-\alpha_2\beta_1\notin 3\mathbb{Z}$; Case 2, $\alpha_1\beta_2-\alpha_2\beta_1\in 3\mathbb{Z}$.

 Case 1, if $\alpha_1\beta_2-\alpha_2\beta_1\notin 3\mathbb{Z}$. Let
\begin{eqnarray*}
A_2=\begin{bmatrix}
 \alpha_1&\beta_1\\
 \alpha_2&\beta_2
\end{bmatrix},  \ \ \ \ \ \ \ \ \ \ M_2=(\alpha_1\beta_2-\alpha_2\beta_1)\begin{bmatrix}
 0&1\\
 1&1
\end{bmatrix}
\end{eqnarray*}
and let $D_2=A_2^{-1}D=\{(0,0)^t,(1,0)^t,(0,1)^t\}$,  $A_2^{-1}M_1A_2=M_2$. Then $M_1$ is an expanding integer matrix with $\gcd(\det(M_1),3)=\gcd(\det(M_2),3)=1$. By \eqref{eq(3.19)} and a direct calculation, we obtain $\mathcal{Z}_{D_2}^2=\{(1/3,2/3)^t,(2/3,1/3)^t\}\subset \mathring{E}_{3}^2$ and $\bigcup_{j=1}^8{M_2^*}^j\mathcal{Z}_{D_2}^2=\mathring{E}_3^2\;({\rm{mod}}  \ \mathbb{Z}^2)$. Hence Lemmas \ref{th(2.2)} and \ref{th(3.7)} imply $n^*(\mu_{M_2,D_2})=n^*(\mu_{M_1,D})=9$, which shows that the number $9$ is the best.

Case 2, $\alpha_1\beta_2-\alpha_2\beta_1\in 3\mathbb{Z}$. Observe that the case (ii) of Theorem \ref{th(1.1)} also satisfies $\alpha_1\beta_2-\alpha_2\beta_1\in 3\mathbb{Z}$, so the following discussions will also be used in the proof of case (ii).

Let $\alpha_1\beta_2-\alpha_2\beta_1=3^{\eta}\gamma$ for some integers $\eta\geq1$ and  $3\nmid \gamma$. Without loss of generality, we assume $\gcd(\alpha_1,\alpha_2)=\sigma$ with $3\nmid\sigma$(Otherwise, we can  choose $\sigma=\gcd(\beta_1,\beta_2)$ with $3\nmid\sigma$, because $\gcd(\alpha_1,\alpha_2,\beta_1,\beta_2)=1$).   Let $\alpha_1=\sigma t_1, \alpha_2=\sigma t_2$, where $\gcd(t_1,t_2)=1$, then there exist integers $p$ and $q$ such that $pt_1+qt_2=1$. Clearly, $\sigma=p\alpha_1+q\alpha_2$ and $\sigma|\gamma$. For convenience, we denote $\omega=p\beta_1+q\beta_2$ and $\vartheta=\gamma/\sigma $. Let
$A_3=\gamma\begin{bmatrix}
 t_1&-q \\
 t_2&p
\end{bmatrix}$. By noting that $t_2\alpha_1=t_1\alpha_2$ and $t_1\beta_2-t_2\beta_1=3^\eta\vartheta$, we have
\begin{eqnarray}\label{eq(3.21)}
 D_3=A_3 ^{-1}D=\frac{1}{\gamma}\begin{bmatrix}
 p&q \\
 -t_2&t_1
\end{bmatrix}\left\{ {\left( {\begin{array}{*{20}{c}}
0\\
0
\end{array}} \right),\left( {\begin{array}{*{20}{c}}
\alpha_1\\
\alpha_2
\end{array}} \right),\left( {\begin{array}{*{20}{c}}
\beta_1\\
\beta_2
\end{array}} \right)} \right\}=\left\{ {\left( {\begin{array}{*{20}{c}}
0\\
0
\end{array}} \right),\left( {\begin{array}{*{20}{c}}
 \frac{1}{\vartheta} \\ 0
\end{array}} \right),\left( {\begin{array}{*{20}{c}}
\frac{\omega}{\sigma \vartheta} \\ \frac{3^\eta}{\sigma}
\end{array}} \right)} \right\}
\end{eqnarray}
and
\begin{eqnarray}\label{eq(3.22)}
M_3=A_3^{-1}MA_3
=\begin{bmatrix}
(pa+qc)t_1+(pb+qd)t_2 & (pb+qd)p-(pa+qc)q \\
(ct_1-at_2)t_1+(dt_1-bt_2)t_2 & (dt_1-bt_2)p-(ct_1-at_2)q
\end{bmatrix}.
\end{eqnarray}
It is obvious that $M_3$ is an expanding integer matrix with $\gcd(\det(M_3),3)=\gcd(\det(M),3)=1$. By Lemma \ref{th(2.2)}, we get $n^*(\mu_{M_3,D_3})=n^*(\mu_{M,D})$.

In view of \eqref{eq(3.19)}, it is easy to calculate that
\begin{align}\label{eq(3.23)}
\mathcal{Z}(m_{D_3})=\{x\in\mathbb{R}^n:m_{D_3}(x)=0\}:=Z_0 \cup \widetilde{Z}_0,
\end{align}
where
\begin{eqnarray*}
Z_0=\left\{
\begin{pmatrix}
  \vartheta(\frac{1}{3}+k_1) \\ \frac{1}{3^{\eta+1}}(2\sigma-\omega-3\omega k_1+3\sigma k_2)
  \end{pmatrix}\
  :k_1,k_2\in\mathbb{Z}
  \right\},
 \end{eqnarray*}
and
\begin{eqnarray*}
\widetilde{Z}_0=\left\{
\begin{pmatrix}
  \vartheta(\frac{2}{3}+\widetilde{k}_1) \\ \frac{1}{3^{\eta+1}}(\sigma-2\omega-3\omega\widetilde{k}_1+3\sigma \widetilde{k}_2)
  \end{pmatrix}\
  :\widetilde{k}_1,\widetilde{k}_2\in\mathbb{Z}
  \right\}.
 \end{eqnarray*}

\begin{pro}\label{th(3.8)}
Let $\alpha_1$, $\alpha_2$, $\beta_1$, $\beta_2$, $\sigma$ and $\omega$ be defined as above. Then

(i) $2 \sigma-\omega \notin3\mathbb{Z}$ if $2\alpha_1-\beta_1\notin 3\mathbb{Z}$ or $2\alpha_2-\beta_2\notin 3\mathbb{Z}$;

(ii) $2 \sigma-\omega \in3\mathbb{Z}$ if $2\alpha_1-\beta_1, 2\alpha_2-\beta_2$
$\in3\mathbb{Z}$.
\end{pro}
\begin{proof}
Note that $\alpha_1=\sigma t_1, \alpha_2=\sigma t_2$, $pt_1+qt_2=1$, $\sigma=p\alpha_1+q\alpha_2$, $t_1\beta_2-t_2\beta_1=3^\eta\vartheta$ and $\omega=p\beta_1+q\beta_2$,
we have
\begin{align}\label{eq(3.24)}
2 \sigma-\omega=p(2\alpha_1-\beta_1)+q(2\alpha_2-\beta_2).
\end{align}
The case (ii) holds from \eqref{eq(3.24)} immediately.

For case (i), without loss of generality, suppose $2\alpha_1-\beta_1\notin 3\mathbb{Z}$. Since $\alpha_1=\sigma t_1$, $\alpha_2=\sigma t_2$ and $t_1\beta_2-t_2\beta_1=3^\eta\vartheta$, we have $t_2\alpha_1=t_1\alpha_2$ and
\begin{align}\label{eq(3.25)}
q3^\eta\vartheta=q(t_1\beta_2-t_2\beta_1)=qt_1(\beta_2-2\alpha_2)+qt_2(2\alpha_1-\beta_1).
\end{align}
Multiplying $2\alpha_1-\beta_1$ on both sides of $pt_1+qt_2=1$ yields $pt_1(2\alpha_1-\beta_1)+qt_2(2\alpha_1-\beta_1)=2\alpha_1-\beta_1$. Combining this with \eqref{eq(3.25)}, we get $t_1(p(2\alpha_1-\beta_1)+q(2\alpha_2-\beta_2))=-q3^\eta\vartheta+2\alpha_1-\beta_1$, which implies $2 \sigma-\omega=p(2\alpha_1-\beta_1)+q(2\alpha_2-\beta_2)\notin 3\mathbb{Z}$, because $2\alpha_1-\beta_1\notin 3\mathbb{Z}$.
\end{proof}

Let $\mathcal {A}_3(s)=\frac{1}{3^s}\{(3^{s-1},l)^t:0\leq l\leq 3^s-1,l\in \mathbb{Z}\}$ be defined by (\ref{eq(1.3)}) and define
\begin{align}\label{eq(3.26)}
\mathcal {A}^i_3(s):=\frac{1}{3^s}\{(3^{s-1},l)^t:0\leq l\leq 3^s-1,l=3k+i, k\in \mathbb{Z}\}
\end{align}
for $i=0,1,2$. Obviously, $\bigcup_{i=0}^{2}\mathcal {A}^i_3(s)=\mathcal {A}_3(s)$.

\begin{pro}\label{th(3.9)}
Let $Z_0$, $\eta$, $\sigma$, $\omega$ and $\vartheta$ be given by (\ref{eq(3.23)}), and let $\mathcal {A}_3$, $\mathcal {A}^i_3$ be defined by (\ref{eq(1.3)}) and \eqref{eq(3.26)}, respectively. Then there exist $\tau\notin3\mathbb{Z}$ such that the following two conclusions hold:

(i) If $2 \sigma-\omega \notin3\mathbb{Z}$, then  ~$\tau\vartheta(\mathcal {A}^i_3(\eta+1)+\mathbb{Z}^2)\subset Z_0$ for  $i=1$ or $2$;

(ii) If $2 \sigma-\omega \in3\mathbb{Z}$, then ~$\tau\vartheta(\mathcal {A}_3(\eta)+\mathbb{Z}^2)\subset Z_0$.
\end{pro}

\begin{proof}
(i) Since $3\nmid\sigma$, let $\sigma=3a+\kappa$ for some integers $a$ and $\kappa=1$ or $2$.
Let $\tau=\kappa\sigma$, we will prove that there exists $i=1$ or $2$ such that ~$\kappa\sigma\vartheta(\mathcal {A}^i_3(\eta+1)+\mathbb{Z}^2)\subset Z_0$. In fact, this only need to show that for any $k'_1,k'_2\in \mathbb{Z}$,  there exist $k_1,k_2\in \mathbb{Z}$ such that
\begin{align}\label{eq(3.27)}
\left\{
\begin{array}{ll}
\kappa\sigma\vartheta (\tfrac{1}{3}+k_1^\prime)= \vartheta(\frac{1}{3}+k_1), \\
\kappa\sigma\vartheta(i+3k_2^\prime)/3^{\eta+1}=(2\sigma -\omega-3\omega k_1+3\sigma k_2)/3^{\eta+1},
\end{array}
\right.
\end{align}
where $i=1$ or $2$ and it is independent of $k'_1,k'_2$. Obviously,
\eqref{eq(3.27)} holds if and only if
\begin{align*}
\left\{
\begin{array}{ll}
k_1= (\kappa\sigma-1)/3+\kappa\sigma k_1', \\
k_2=(\kappa \vartheta i+\kappa \omega-2)/3+\kappa \omega k'_1+\kappa \vartheta k_2'.
\end{array}
\right.
\end{align*}
Since $\kappa=1$ or $2$ and $\kappa\sigma=\kappa(3a+\kappa)=3a\kappa+\kappa^2=\kappa^2=1\;({\rm mod} \ 3)$, then $k_1\in\mathbb{Z}$. We then prove that $k_2$ is also an integer. It follows from $\sigma=3a+\kappa$ and $\kappa^2=1\;({\rm mod} \ 3)$ that $\kappa(2 \sigma-\omega)=6a\kappa+2\kappa^2-\kappa \omega=2-\kappa \omega\;({\rm mod} \ 3)$. This together with  $2 \sigma-\omega,  \kappa \notin3\mathbb{Z}$ implies $2-\kappa \omega\notin 3\mathbb{Z}$. Therefore, by noting that $\kappa,\vartheta \notin3\mathbb{Z}$, we can always choose $i=1$ or $2$ such that $\kappa \vartheta i=2-\kappa \omega\;({\rm mod} \ 3)$, which shows that $k_2\in \mathbb{Z}$. Hence \eqref{eq(3.27)} holds and ~$\tau\vartheta(\mathcal {A}^i_3(\eta+1)+\mathbb{Z}^2)\subset Z_0$ for $i=1$ or $2$.

(ii) Let $\tau=\kappa\sigma$ as case (i), we will show that~$\kappa\sigma\vartheta(\mathcal {A}_3(\eta)+\mathbb{Z}^2)\subset Z_0$. Indeed, we only need to show that for any $k_1^\prime,k_2^\prime\in\mathbb{Z}$, there exist $k_1,k_2\in\mathbb{Z}$ such that \eqref{eq(3.27)} holds for $i=0$, i.e.,
\begin{align}\label{eq(3.28)}
\left\{
\begin{array}{ll}
\kappa\sigma\vartheta (\tfrac{1}{3}+k_1^\prime)=\vartheta(\tfrac{1}{3}+k_1),\\
\kappa\sigma\vartheta\cdot k_2^\prime/3^\eta=(2\sigma-\omega-3\omega k_1+3\sigma k_2)/3^{\eta+1}.
\end{array}
\right.
\end{align}
\eqref{eq(3.28)} holds if and only if
\begin{align*}
\left\{
\begin{array}{ll}
k_1= (\kappa\sigma-1)/3+\kappa\sigma k_1', \\
k_2=(\kappa \omega-2)/3+\kappa \omega k'_1+\kappa \vartheta k_2'.
\end{array}
\right.
\end{align*}
Using the similar proof as case (i), we obtain $k_1,k_2\in\mathbb{Z}$ by $\kappa\sigma=1\;({\rm mod} \ 3)$, $2 \sigma-\omega\in 3\mathbb{Z}$ and $0=\kappa(2 \sigma-\omega)=2-\kappa \omega\;({\rm mod} \ 3)$. Hence (\ref{eq(3.28)}) holds and $\tau\vartheta(\mathcal {A}_3(\eta)+\mathbb{Z}^2)\subset Z_0$.
\end{proof}

We now continue to prove Case 2 of Theorem \ref{th(1.1)}(i). By Propositions \ref{th(3.8)}(i) and \ref{th(3.9)}(i), there exists $\tau\notin3\mathbb{Z}$ such that~$\tau\vartheta(\mathcal {A}^i_3(\eta+1)+\mathbb{Z}^2)\subset Z_0$ for $i=1$ or $2$.

Now considering an integer matrix $M_3$ in \eqref{eq(3.22)} with its transposed conjugate
\begin{align}\label{eq(3.29)}
M_3^{*}=\bar{N}\begin{bmatrix}
0 & 1\\
1 & 1
\end{bmatrix}:=\bar{N}\hat{M}\in GL(2,\mathbb{F}_3),
\end{align}
where $\bar{N}=3^2n+1$ is a sufficient large integer such that $M_3$ is an expanding integer matrix. We will show that $M_3$ satisfies $n^*(\mu_{M_3,D_3})=9$. It is easy to check that $O_3(M_3^{*})=O_3(\hat{M})= 8$ and ${M_3^{*}}^8=(3^2n+1)^8\left(3\begin{bmatrix}
4 & 7\\
7 & 11
\end{bmatrix}+I\right):=3\begin{bmatrix}
l_1 & l_2\\
l_3 & l_4
\end{bmatrix}+I$. Similar to \eqref{eq(3.16)}, it is easy to check that $l_i\notin 3\mathbb{Z}$ for all $1\leq i\leq 4$, hence Theorem \ref{th(3.3)} implies that $O_{3^{\eta+1}}(M_3^{*})=8\cdot3^{\eta}$ and $\bigcup_{j=1}^{\infty} M_3^{*j}\xi=\bigcup_{j=1}^{8\cdot 3^{\eta}} M_3^{*j}\xi\;({\rm mod}~\mathbb{Z}^2)$ for any $\xi\in \mathring{E}_{3^{\eta+1}}^2$. Let $\det (M_3^{*})=L$ and $\varphi$ be the Euler's phi function. Note that $\gcd(L,3)=1$, $\varphi(3^{\eta+1})=2\cdot3^{\eta}$ and~$\tau\vartheta(\mathcal {A}^i_3(\eta+1)+\mathbb{Z}^2)\subset Z_0\subset\mathcal{Z}(m_{D_3})$ for $i=1$ or $2$, by using the similar proof as \eqref{eq(3.14)}, we obtain
\begin{align} \label{eq(3.30)}
\mathcal{ Z}(\hat{\mu}_{M_3,D_3}) \supset  \bigcup_{j=1}^{8\cdot 3^{\eta}}L^{2\cdot3^{\eta}\cdot j}\tau\vartheta(M_3^{*j}\mathcal {A}^i_3(\eta+1)+\Bbb Z^2).
\end{align}
Note that $M_3^*$ satisfies Theorem \ref{th(3.4)}, according to \eqref{eq(3.5)}, \eqref{eq(3.10)} and \eqref{eq(3.26)}, for any $i=1$ or $2$, we have
\begin{equation}\label{eq(3.31)}
 \bigcup_{j=1}^{8\cdot 3^{\eta}}M_3^{*j}\mathcal {A}^i_3(\eta+1)=\bigcup_{l\in \mathcal{R}_{3,\eta+1}(i)}\mathcal{B}_{3,\eta+1}(l)= \mathring{E}_{3^{\eta+1}}^2\setminus\mathring{E}_{3^{\eta}}^2\ \;({\rm mod} \; \Bbb Z^2).
\end{equation}
Let $\Lambda=L^{ 2\cdot 3^{\eta} \cdot8\cdot 3^{\eta}} \tau\vartheta E$, where
\begin{equation*}
E=\frac{1}{3^{\eta+1}}\left\{\begin{pmatrix}
  0 \\ 0
  \end{pmatrix},\begin{pmatrix}
  0 \\ 1
  \end{pmatrix},\begin{pmatrix}
 0 \\ 2
  \end{pmatrix},\begin{pmatrix}
  1 \\ 0
  \end{pmatrix},\begin{pmatrix}
  1 \\ 1
  \end{pmatrix},\begin{pmatrix}
 1 \\ 2
  \end{pmatrix},\begin{pmatrix}
  2 \\ 0
  \end{pmatrix},\begin{pmatrix}
  2 \\ 1
  \end{pmatrix},\begin{pmatrix}
  2 \\ 2
  \end{pmatrix}\right\}.
\end{equation*}
We will show that $\mathcal {E}_\Lambda=\{e^{2\pi i\langle\lambda,x\rangle}:\lambda\in\Lambda\}$ is an orthogonal set of $L^2(\mu_{M_3,D_3})$. For any $\lambda_1\ne \lambda_2\in\Lambda$, there exists $\lambda^\prime\in\mathring{E}_{3^{\eta+1}}^2\setminus\mathring{E}_{3^{\eta}}^2\ ({\rm mod}\; \Bbb Z^2)$ such that $\lambda_1-\lambda_2= L^{ 2\cdot 3^{\eta} \cdot8\cdot 3^{\eta}}\tau\vartheta\lambda^\prime$. By~$(\ref{eq(3.31)})$, there exist $\lambda_0 \in \mathcal {A}^i_3(\eta+1)$ and $1\leq j_0\leq 8\cdot 3^{\eta}$ such that $\lambda^\prime=M_3^{*j_0}\lambda_0\;({\rm mod}\; \Bbb Z^2)$. Using the similar proof as in the last part of Theorem \ref{th(1.2)} and~$(\ref{eq(3.30)})$, we have
\begin{align*} \nonumber
\lambda_1-\lambda_2 &\in L^{ 2\cdot 3^{\eta} \cdot8\cdot 3^{\eta}}\tau\vartheta(M_3^{*j_0}\lambda_0+\Bbb Z^2)\subset L^{ 2\cdot 3^{\eta} \cdot8\cdot 3^{\eta}}\tau\vartheta(M_3^{*j_0}\mathcal {A}^i_3(\eta+1)+\Bbb Z^2)\\
&\subset L^{ 2\cdot 3^{\eta}j_0}\tau\vartheta(M_3^{*j_0}\mathcal {A}^i_3(\eta+1)+\Bbb Z^2)\subset \mathcal{ Z}(\hat{\mu}_{M_3,D_3}).
\end{align*}
This shows that $(\Lambda-\Lambda)\setminus\{0\}\subset\mathcal{ Z}(\hat{\mu}_{M_3,D_3})$. It follows from~$(\ref{eq(2.2)})$ that the elements in $\mathcal {E}_\Lambda$ are mutually orthogonal, and hence $n^*(\mu_{M_3,D_3})\geq 9$. Combining with $n^*(\mu_{M_3,D_3})=n^*(\mu_{M,D})\leq 9$, we obtain $n^*(\mu_{M,D})=9$ and complete the proof of Theorem \ref{th(1.1)}(i).
\medskip

(ii) Suppose $2\alpha_1-\beta_1, 2\alpha_2-\beta_2\in3\mathbb{Z}$, then $\alpha_1\beta_2-\alpha_2\beta_1\in 3\mathbb{Z}$. We mainly use Theorem \ref{th(1.2)} to complete the proof. By the same transformation as Case 2 of (i), we get the same matrix $M_3$ and digit set $D_3$ in \eqref{eq(3.21)} and \eqref{eq(3.22)}. Thus $n^*(\mu_{M_3,D_3})=n^*(\mu_{M,D})$, and the zero set of $m_{D_3}(x)$ is also (\ref{eq(3.23)}). By Proposition \ref{th(3.8)}(ii), $\vartheta\notin 3\mathbb{Z}$ and \eqref{eq(3.23)},  we have
 \begin{eqnarray}\label{eq(3.32)}
 \mathcal{Z}(m_{D_3})\subset \mathring{E}_{3^{\eta}}^2+\mathbb{Z}^2.
 \end{eqnarray}
At the same time, it follows from Propositions \ref{th(3.8)}(ii) and \ref{th(3.9)}(ii) that there exists $\tau\notin3\mathbb{Z}$ such that~$\tau\vartheta(\mathcal {A}_3(\eta)+\mathbb{Z}^2)\subset Z_0$. Then (\ref{eq(3.23)}) and (\ref{eq(3.32)}) imply that $\tau\vartheta(\mathcal {A}_3(\eta)+\mathbb{Z}^2)\subset \mathcal{ Z}(m_{D_3})\subset \mathring{E}_{3^\eta}^2+\mathbb{Z}^2$, where $\tau\vartheta \notin 3\mathbb{Z}$. By \eqref{eq(3.22)}  and  the definition of matrix $A_3$, it is easy to see that $M_3$ can be any expanding integer matrix with $\gcd(\det(M_3),3)=1$. Therefore, by Theorem \ref{th(1.2)},  $n^*(\mu_{M,D})=n^*(\mu_{M_3,D_3})\leq3^{2\eta}$ and the number $3^{2\eta}$ is the best.

This completes the proof of Theorem \ref{th(1.1)}.
\end{proof}

\end{document}